\documentclass{article}
\usepackage{graphicx,amssymb,amsmath,amsthm}
\usepackage{lineno}
\usepackage[usenames,dvipsnames]{xcolor}

\usepackage{float}

\usepackage{tikz}
\usetikzlibrary{decorations.markings,arrows,backgrounds,positioning,fit,petri,shapes,decorations.pathreplacing,positioning, arrows.meta}
\usepackage{capt-of}
\usepackage{enumerate}
\usepackage{hyperref}
\hypersetup{colorlinks=true , allcolors=blue}

\usepackage{afterpage}

\usepackage{mathtools}
\usepackage{tikz-cd}
\usepackage{titlesec}
\usepackage{ mathrsfs}

\usepackage{ marvosym }

\usepackage{pgfplots}
\pgfplotsset{compat=1.15}
\usepackage{mathrsfs}
\usetikzlibrary{arrows}
\newtheorem{theorem}{Theorem}[section]

\newtheorem{lemma}[theorem]{Lemma}

\title{A Note on the Asymptotic Value  of the Isoperimetric Number of $ J(n,2) $\footnote{Supported by CONACYT FORDECYT-PRONACES/39570/2020}}
\author{Ruy Fabila-Monroy\thanks{Departamento de Matem\'aticas, CINVESTAV.} \thanks{{\tt ruyfabila@math.cinvestav.edu.mx}}    \and Daniel Gregorio-Longino\footnotemark[2] \thanks{{\tt dgregorio@math.cinvestav.mx}}}

\begin{document}
\maketitle

\begin{abstract}
Let $G$ be a graph on $n$ vertices and $S$ a subset of vertices of $G$; the boundary of $S$ is the set,  $\partial S$, of edges of $G$ connecting $ S $ to its complement in $G$.
The isoperimetric number of $G$, is the minimum of 
$\left| \partial S \right|/\left| S \right|$ overall $S \subset V(G)$ of at most $n/2$ vertices. Let $k \le n$ be positive
integers. The Johnson graph is the graph, $J(n,k)$, whose vertices are all the subsets of size $k$ of $\{1,\dots,n\}$, two
of which are adjacent if their intersection has cardinality equal to $k-1$. 
In this paper we show that the asymptotic value of the isoperimetric number of the Johnson graph $J(n,2)$ is equal
to $ (2-\sqrt{2})n$.
\end{abstract}

\section{Introduction}

Throughout this paper let $G$ be a graph on $n$ vertices. Let $S \subset V(G)$;
 the \textit{boundary} of $S$ is the set,  $\partial S$, of edges of $G$ connecting $ S $ to its complement $ S^{c} $. Note that $ S $ and $ S^{c} $ have the same boundary.
 The \textit{isoperimetric number} of $ G $, is the number 
 \begin{equation*}
\operatorname{iso}(G) := \min \left \{\frac{ \left| \partial S \right|}{\left| S \right|}: S \subset V(G) \textrm{ and } 0 < |S| \le  \frac{n}{2} \right \}.
\end{equation*} 
A subset of vertices, on which this minimum is attained, is called an \textit{isoperimetric set} of $ G. $\

 The isoperimetric number of $G$ serves as a measure of the connectivity of $G$; 
 it quantifies the minimal interaction between a subset of vertices and its complement, in terms of the number of edges between them. A quantity related to $\operatorname{iso}(G)$ is the bisection width. The \textit{bisection width} of a graph $ G $ is the minimum number, $\operatorname{bw}(G)$,  of edges which must be removed from $G$ in order to split its vertices into equal-sized sets(with a difference of one, when the number of vertices of $ G $ is odd); see \cite{goldberg1984minimal}. 
 Note that \[\operatorname{bw}(G) \ge \operatorname{iso}(G) \left \lfloor \frac{n}{2} \right \rfloor.\]

 The quantity $\operatorname{iso}(G)$ is also a discrete analogue of the Cheeger isoperimetric constant.
 Given a Riemannian $d$-manifold $M$, the  Cheeger isoperimetric constant is defined as
 \[\inf_A \frac{\operatorname{Vol}_{d-1}\partial (A)}{\operatorname{Vol}_d (A)}, \]
 where $A$ is taken over all non-empty open bounded subsets of $M$. 
 This analogy is described in more detail in \cite{buser1978cubic,buser1984bipartition}.

Let $ n $ and $ k $ be positive integers with $ k\leq n $. Let $[n]:=\{1,2,\dots,n\}$, and $\binom{[n]}{k}$ be the set
of all subsets of $[n]$ of $k$ elements.  The \textit{Johnson graph} $ J(n,k) $ is the graph whose vertex set is $\binom{[n]}{k}$;
where two vertices of $J(n,k)$ are adjacent if their intersection has cardinality equal to $ k-1 $. The Johnson graph has been widely studied due to its connections with coding theory; see for example \cite{etzion1996chromatic, graham1980lower,johnson1962new, liebler2014neighbour,macwilliams1977theory, neunhoffer2014sporadic}.\

Let $ v = \{v_{1},\ldots, v_{k}\} $ be a vertex of $ J(n,k) $. For convenience we always
assume that $ v_{i}<v_{i+1} $, for $ i=1,\ldots , k-1$. Let $w=\{w_1,\ldots,w_k\}$ be another vertex of $J(n,k)$.
In the \textit{lexicographic order}, $v$ is greater than $w$ if $ v_{1}>w_{1} $ or there exists $ 2\leq i \leq k $ such that $ v_{j}=w_{ j} $ for $ 1\leq j< i $ and $ v_{i}>w_{i} $. For example, in $ J(4,2) $ the vertex $ \{1,3\} $ is greater in the lexicographic order than the vertex $ \{1,2\} $. In what follows we assume that the vertices of $J(n,k)$ are ordered by the lexicographic order. 
Let $ m $ be a positive integer; let $ F_{m} $ be the set of the first $ m $ vertices of $J(n,k)$, and 
let $ L_{m} $ be the be set of the last $ m $ vertices of $J(n,k)$.
A classic isoperimetric problem is to  determine the number   
\[  B_{G}(m) = \min \{\left| \partial S \right|: S\subseteq V(G) \textrm{ and } \left| S \right|= m\}.\]  Ahlswede and Katona solved this problem for the Johnson graph $J(n,2)$. \begin{theorem}\label{kat}(Ahlswede and Katona \cite{ahlswede1978graphs})
For every $ 0<m\leq n(n-1)/2 $, the value of $ B_{J(n,2)}(m) $ is attained by the set $ F_{m} $ or by $ L_{m} $.
\end{theorem} In this note we use Theorem \ref{kat} to give the asymptotic value of $ \operatorname{iso} (J(n,2)). $  
\begin{theorem}\label{main}
\begin{equation*}
  \operatorname{iso}( J(n,2))\sim \left (2-\sqrt{2} \right )n.
\end{equation*}
\end{theorem}

\section{Preliminaries}

In what follows let 
\[n^\ast:= \lfloor n(n-1)/4\rfloor. \]
By  Theorem \ref{kat}, if $ S $ is an isoperimetric set of $J (n,2) $, then $ S $ is equal to one of 
$ F_{1},\ldots , F_{n^\ast},L_{1},\ldots, L_{n^\ast} $. For the proof of Theorem~\ref{main}, we minimize the expression $ |\partial S| / |S| $ over these sets.

Before proceeding, we give an expression for $ |\partial S| $. This result is stated in \cite{bey2006remarks}. We provide a short proof for completeness.

\begin{lemma}\label{iso_s}
Let $ S $ be a non-empty subset of vertices of $ J(n,k) $. 
For  every $A \in \binom{[n]}{k-1}$,
let $ S_{A}=\{v\in S : A\subset v \}$.
Then \begin{equation*}
\begin{split}
|\partial S|  &= k(n-k+1)|S|-\displaystyle\sum_{A \in \binom{[n]}{k-1}}{|S_{A }|^2}.
\end{split}
\end{equation*}
\end{lemma}

\begin{proof}

Let $ v=\{v_{1},v_{2},\ldots,v_{k}\} \in S $. For every $1 \le i \le k$, let $ A_{v,i}:=v\setminus\{v_{i}\}$.
Note that $v$ is adjacent to $n-k$ vertices of the form $ A_{v,i}\cup \{w\} $.
Since $v \in S_{A_{v,i}}$, exactly $|S_{A_{v,i}}|-1$ of these vertices
are in $S$. 
Therefore, each $v \in S$ contributes with \begin{equation*}
\begin{split}
\displaystyle\sum_{i=1}^{k}\left [ n-k-(|S_{A_{v,i}}|-1) \right ] &= k(n-k+1)-\displaystyle\sum_{i=1}^k{|S_{A_{v,i}}|}
\end{split}
\end{equation*} edges to the boundary of $S$.
Therefore,
\[|\partial S|= k(n-k+1)|S|-\sum_{v \in S}\sum_{i=1}^k{|S_{A_{v,i}}|}.\]

Let  $A \in \binom{[n]}{k-1}$. Note that $v \in S_A$ if and only
if $A=A_{v,i}$, for some $i \in [k]$. Thus, $A$ appears
exactly $|S_A|$ times as an $S_{A_{v,i}}$ in $\sum_{v \in S}\sum_{i=1}^k{|S_{A_{v,i}}|}.$
This implies that 
\[\sum_{v \in S}\sum_{i=1}^k{|S_{A_{v,i}}|}=\sum_{A \in \binom{[n]}{k-1}} |S_A|^2.\]
The result follows.
\end{proof}


In what follows,  to represent the vertices of $ J(n,2) $  we use triangular diagrams like the one of Figure~\ref{diagrama}.
\begin{figure}[H]
\begin{center}
\begin{tikzpicture}[inner sep=0.5 mm,xscale=0.5,yscale=.5]

\node (c) at (0,0) {$\{1,2\}$};
\node (c) at (2,0) {$\{1,3\}$};
\node (c) at (4,0) {$\ldots$};
\node (c) at (6,0) {$\{1,n\}$};

\node (c) at (2,-1) {$\{2,3\}$};
\node (c) at (4,-1) {$\ldots$};
\node (c) at (6,-1) {$\{2,n\}$};

\node (c) at (6,-2) {$\vdots$};
\node (c) at (6,-3) {$\{n-1,n\}$};

\end{tikzpicture}
\end{center}
\caption{The vertices of $J(n,2)$ arranged in a triangular diagram }\label{diagrama}
\end{figure}
We use the usual notions of direction(left, right, up and down), to indicate the position of one vertex with respect to another in the diagram. For example, vertex $ \{2,3\} $ is below vertex $ \{1,3\} $ and to the left of vertex $ \{2,4\} $. 
Let $v=\{i,j\} \in V(J(n,2))$. We say that $ v $ is in row $i $ and in column $j$. Note that  the neighbors of $v $ can be visualized in a geometric  way, they are: all the different vertices from $ v $ that are in the rows $ i $ and $ j $; and all the different
vertices from $ v $ that are in the columns $ i $ and $ j $. Also note that the vertices of $J(n,2)$ that are greater (in lexicographic order) than $v$ are precisely the vertices in the diagram that are to the right or below $v$.

Let $v=\{i,j\}$ and $w=\{i',j'\}$ be vertices of $J(n,2)$. We say that $v$ \emph{dominates}
$w$, if $i \le i'$ and $j \le j'$.
A subset $ S $ of vertices of $ J(n,2) $ is \textit{stable} (see \cite{harper1977stabilization,harper1991problem}) 
if for every $v \in S$  all the vertices of $ J(n,2) $ dominated by $v$ are also in $S$. 
We say that a stable set $S$ is \emph{determined by} a set $X \subset V(J(n,2))$, if $S$ consists of all the vertices of $J(n,2)$ that are dominated by a vertex in $X$.

\begin{figure}
 	\centering
 	\includegraphics[width=0.4\textwidth]{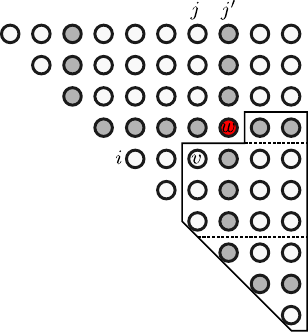}
 	\caption{Depiction of the proof of Lemma~\ref{horizontal} $w$ is painted red; its neighbors are shaded and
 	$S$ is enclosed with a polygon.}
 	\label{fig:hor}
 \end{figure}

We first show that under certain conditions,  it is possible to add a vertex to a stable set to obtain a stable set $ S' $ satisfying that $ |\partial S'|/| S'|\leq |\partial S|/| S| $.
\begin{lemma}\label{horizontal}
Let  $v=\{i,j\}$ and $w=\{i-1,j'\}$ be two vertices of $J(n,2)$, with $j'\ge j$. Let $S$ be the union
of the set of vertices to the right of $w$ and the set of vertices dominated by $v$. 
 Then $ S':=S\cup \{w\} $ is a stable set, and
\[\frac{|\partial S'|}{|S'|}\leq \frac{|\partial S|}{|S|}. \]
\end{lemma}
\begin{proof}
The vertices $\{1,j'\},\dots,\{i-2,j'\}$ ,$\{1,i-1\},\dots,\{i-2,i-1\}$, $\{i-1,i\},\dots,\{i-1,j'-1\}$  are the neighbors of $w$ that are not in $S$; and
the vertices $\{i-1,j'+1\},\dots,\{i-1,n\}$, $\{i,j'\},\dots \{j'-1,j'\}$, $\{j',j'+1\},\dots,\{j',n\}$ are the neighbors of $w$ that are in $S$. See Figure~\ref{fig:hor}. Therefore,
\begin{align*}
 |\partial S'|& =|\partial S|+(i-2)+(i-2)+(j'-i)-(n-j')-(j'-i)-(n-j') \\
 & =|\partial S|+2(i+j'-n -2).
\end{align*}

$S$ is formed by: the vertices in row $i-1$ to the right of $w$; a rectangle of sides equal to $j-i$ and $(n-j+1)$; and 
a triangle of base $(n-j)$. See Figure~\ref{fig:hor}. Therefore, 
\[|S|=(n-j')+(j-i)(n-j+1)+\frac{(n-j+1)(n-j)}{2}.\]

By Lemma~\ref{iso_s} we have that
\begin{align*}
|\partial S|& =2(n-1)|S|-\sum_{l=i-1}^n \left |S_{\{l\}} \right |^2 \\
&=2(n-1)|S|-(n-j')^2-\sum_{l=i}^{j-1}(n-j+1)^2-\sum_{l=j}^{j'}(n-i)^2 \\ & -\sum_{l=j'+1}^{n}(n-i+1)^2 \\
&=2(n-1)|S|-(n-j')^2-(j-i)(n-j+1)^2-(j'-j+1)(n-i)^2\\ & -(n-j')(n-i+1)^2
\end{align*}
Note that $|S'|=|S|+1$. 

Putting everything together we have that
\begin{align*}
 &\frac{|\partial S'|}{|S'|} \le\frac{|\partial S|}{|S|} \\
 & \iff \\
 & \frac{|\partial S|+2(i+j'-n -2)}{|S|+1} \le\frac{|\partial S|}{|S|}\\
 & \iff \\
 &2|S|(i+j'-n-2) \le |\partial S| \\
 &\iff \\
 & 2\left ((n-j')+(j-i)(n-j+1)+\frac{(n-j+1)(n-j)}{2} \right )(i+j'-n-2) \\ & \le 2(n-1)|S|-(n-j')^2-(j-i)(n-j+1)^2-(j'-j+1)(n-i)^2\\ & -(n-j')(n-i+1)^2 \\
 & \iff \\
 & 0 \le n^3-j'n^2+3n^2+ni^2+2j'n i-3j'n-3n i+2n-ji^2+i^2-2j'j i-2in^2 \\ &+2i j^2
 -ji+2j'i-i+j'j^2-j'j-j^3+j'^2+j^2-j'
\end{align*}
The last term can be rewritten as
\begin{align*}
  & (n-j)^2+(n-j)(j-i)^2+2(n-j)(j-i)(n-j')
 +(n-j)(j-i) \\ & +(n-j)(n-j')+(n-j)+(j-i)^2+2(j-i)(n-j')
 +(j-i)+(n-j')^2 \\ & +(n-j')(n-j)^2+(n-j');
\end{align*}
each of the terms in this sum is greater or equal to zero. The result follows.
\end{proof}

\begin{figure}
 	\centering
 	\includegraphics[width=0.4\textwidth]{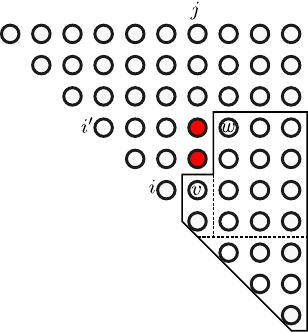}
 	\caption{Depiction of $S$ and $Q$ in Lemma~\ref{vertical}; $S$ is enclosed with a polygon and $Q$
 	is painted red.}
 	\label{fig:ver}
 \end{figure}

For certain types of stable sets $ S $ adding a certain vertex $ w $ does not guarantee that a new obtained set $ S' $ satisfies that $ |\partial S'|/| S'|\leq |\partial S|/| S| $, but rather a certain set of vertices that are on the same column must be added.

\begin{lemma}\label{vertical}
Let  $v=\{i,j\}$ and  $w=\{i',j+1\}$ be two vertices of $J(n,2)$, with $i' < i$. Let $S$ be the set of vertices determined by
$\{v,w\}$. Let $Q$ be the set of vertices in column $j$, and between rows $i'$ and $i$; that is 
$Q=\{\{i',j\},\{i'+1,j\},\dots \{i-1,j\}\}$
 Then $ S':=S\cup Q$ is a stable set, and
\[\frac{|\partial S'|}{|S'|}\leq \frac{|\partial S|}{|S|}. \]
\end{lemma}
\begin{proof}
 Note that $S$ consists of: a triangle of base equal to $n-j$; a rectangle of base equal to $n-j$ and height equal to $j-i'$;
 and a rectangle of base equal to one and height equal to $j-i$. See Figure~\ref{fig:ver}.
 Therefore,
 \[|S|=\frac{(n-j+1)(n-j)}{2}+(n-j)(j-i')+j-i.\]
 We also have that 
 \[|S'|=|S|+(i-i')=\frac{(n-j+1)(n-j)}{2}+(n-j+1)(j-i').\]
 
 By Lemma~\ref{iso_s}, we have that
 \begin{align*}
  |\partial S|& =2(n-1)|S|-\sum_{l=i'}^n \left |S_{\{l\}} \right |^2\\
  & =2(n-1)|S|-\sum_{l=i'}^{i-1} \left |S_{\{l\}} \right |^2-\sum_{l=i}^{j-1} \left |S_{\{l\}} \right |^2-\left |S_{\{j\}} \right |^2-\sum_{l=j+1}^n \left |S_{\{l\}} \right |^2\\
  &=2(n-1)|S|-\sum_{l=i'}^{i-1} (n-j)^2-\sum_{l=i}^{j-1} (n-j+1)^2-(n-i)^2-\sum_{l=j+1}^n (n-i')^2\\
  &=2(n-1)|S|-(i-i') (n-j)^2-(j-i)(n-j+1)^2-(n-i)^2\\& -(n-j) (n-i')^2
 \end{align*}
 and 
\begin{align*}
  |\partial S'|& =2(n-1)|S'|-\sum_{l=i'}^n \left |S'_{\{l\}} \right |^2\\
  & =2(n-1)|S'|-\sum_{l=i'}^{j-1} \left |S_{\{l\}} \right |^2-\sum_{l=j}^{n} \left |S_{\{l\}} \right |^2\\
  & =2(n-1)|S'|-\sum_{l=i'}^{j-1}(n-j+1)^2-\sum_{l=j}^{n} (n-i')^2\\
  & =2(n-1)|S'|-(j-i')(n-j+1)^2-(n-j+1)(n-i')^2
 \end{align*}
Then 

\begin{align*}
 &\frac{|\partial S'|}{|S'|} \leq\frac{|\partial S|}{|S|} \\
 & \iff \\
 & 2(n-1)-\frac{(j-i')(n-j+1)^2+(n-j+1)(n-i')^2}{\frac{(n-j+1)(n-j)}{2}+(n-j+1)(j-i')}\leq \frac{|\partial S|}{|S|}\\
& \iff \\
& \frac{(j-i')(n-j+1)+(n-i')^2}{\frac{(n-j)}{2}+(j-i')}\\
& \geq  \frac{(i-i') (n-j)^2+(j-i)(n-j+1)^2+(n-i)^2 +(n-j) (n-i')^2}{\frac{(n-j+1)(n-j)}{2}+(n-j)(j-i')+j-i} \\
& \iff \\
&-\frac{ji'^2}{2}-\frac{ji^2}{2}-\frac{ni^2}{2}+i'ji-i'in+i'i^{2}+in^2-i'n^2-ii'^2\\
&+\frac{3ni'^2}{2}+\frac{ji'}{2}+\frac{in}{2}-\frac{i'n}{2}-\frac{ji}{2}\geq 0
\end{align*}
The last term can be rewritten as
\begin{align*}
  & (n-j)^2(i-i')+\frac{(n-j)(i-i')(2j-i-i')}{2}+(n-j)(j-i')(i-i')\\
  &+\frac{(n-j)(i-i')}{2}+(j-i)(j-i')(i-i')
\end{align*}
each of the terms in this sum is greater or equal to zero. The result follows.
\end{proof}

\begin{lemma}\label{lower_bound}
Let $v$ be the last vertex of $ F_{n^\ast} $; let $F'$ be  the set obtained from $ F_{n^\ast} $ by adding all the
vertices to right of $v$ and  in the same row as $v$. Then 
\begin{equation}
\textup{iso } J(n,2)\geq \min \{|\partial L_{n^\ast}|/ | L_{n^\ast}|, |\partial F'|/ | F'|\}
\end{equation}
\end{lemma}
\begin{proof}
Let $S$ be a subset of vertices of $ J(n,2) $ of cardinality $n'\le n^\ast$. 
By Theorem \ref{kat}  we have that 
\[ |\partial S| \geq |\partial L_{n'}| \textrm{ or  } |\partial S|\geq |\partial F_{n'}|.\]
By iteratively applying Lemma~\ref{horizontal} , we have that 
\[|\partial L_{i+1}|/|L_{i+1}| \le |\partial L_{i}|/|L_{i}|\]
 for all  $1 \le i \le n^\ast.$  Therefore,
\[|\partial L_{n'}|/|L_{n'}|\geq |\partial L_{n^\ast}|/|L_{n^\ast}|.\]

Let $\ell$ be the straight line in our diagram, passing through the points $\{1,n\}$ and $\{2,n-1\}$. Let $f$ be the map
that reflects every point $\{i,j\}$ over $\ell$. We have that $f$ is an isomorphism of $J(n,2)$. 
Let \[R_1:=f(F_{n'}) \textrm{ and } R_2:=f(F').\]
By Lemma~\ref{vertical}, we have that
\[\frac{|\partial R_2|}{|R_2|} \le \frac{|\partial R_1|}{|R_1|}. \]
Since $f$ is an isomorphism we have that 
\[\frac{|\partial F'|}{|F'|}\le \frac{|\partial F_{n'}|}{|F_{n'}|}. \]
Therefore,
\[\frac{|\partial S |}{|S|} \ge \frac{|\partial L_{n^\ast} |}{| L_{n^\ast}|} \textrm{ or } \frac{|\partial S |}{|S|} \ge \frac{|\partial F' |}{| F'|}. \]
The result follows.
\end{proof}
We are ready to prove Theorem \ref{main}.

\section{Proof of Theorem \ref{main}}
\begin{proof}[Proof of Theorem \ref{main}]

Let $p$ be the smallest row such that number of vertices in the first $p$ rows is greater or equal to
$n^\ast$. Note that for every $1 \le i \le n-1$ the number of elements in the first $i$
rows is equal to $in-{i(i+1)}/{2}.$ Therefore,
\[p=\left \lceil \frac{2n-1-\sqrt{(2n-1)^2-8n^\ast}}{2} \right \rceil.\]
Let $F'$ be as in Lemma~\ref{lower_bound}; by Lemma~\ref{iso_s} we have that 
  \[\frac{|\partial F'|}{|F'|} = 2(n-1)-\dfrac{p(n-p)+(n-1)^2}{n-(p+1)/2}.\]

Let $q$  be the smallest integer such that the number of vertices in  the last $q$ rows is greater or equal
 to $n^\ast$. We have that for every $1 \le i \le n-1$, the number of vertices in the last $i$ rows is equal to 
  $i(i+1)/2$. Therefore,
 \[q= \left \lceil \frac{\sqrt{8n^\ast+1}-1}{2} \right \rceil.\]
Let $v=\{i,j\}$ be such that $L_{n^\ast}$ is the set of vertices greater or equal to $v$.
Note that $i=n-q$ and that row $i$ has exactly $q$ vertices. Thus, the number of vertices greater
or equal to $\{i,j\}$ is equal to $(n-j+1)+q(q-1)/2$. Equating this to $n^\ast$ we have that
\[j=n+\frac{q(q-1)}{2}+1-n^\ast.\]

Note that $L_{n^\ast}$ consists of $n-j+1$ vertices in row $i$ and a triangle of base equal to $q-1$ and height equal to $q-1$.
Therefore, 
\begin{align*}
 |L_{n^\ast}|&=n-j+1+\frac{q(q-1)}{2}\\
 &=n-j+1+\frac{(n-i)(q-1)}{2}\\
 &=\frac{qn+n-2j-qi+i+2}{2}.
 \end{align*}
 
Note that: $ |{L_{n^\ast}}_{\{l\}} | = 0$ for $ l \leq i $; $ |{L_{n^\ast}}_{\{i\}} | = n-j+1$; $ |{L_{n^\ast}}_{\{l\}} | = q-1$ for $i+1 \le l \le j-1$;  and $ |{L_{n^\ast}}_{\{l\}} | = q$ for  $ l \ge j $.
By Lemma~\ref{iso_s} we have that
\begin{align*}
  |\partial L_{n^\ast}|& =2(n-1)|L_{n^\ast}|-\sum_{l=1}^n \left |{L_{n^\ast}}_{\{l\}} \right |^2\\
  & =2(n-1)|L_{n^\ast}|-|{L_{n^\ast}}_{\{i\}}|^2-\sum_{l=i+1}^{j-1} \left |{L_{n^\ast}}_{\{l\}} \right |^2-\sum_{l=j}^{n} \left |{L_{n^\ast}}_{\{l\}} \right |^2\\
& =2(n-1)|L_{n^\ast}|-(n-j+1)^{2}-(q-1)^{2}(j-i-1)-q^{2}(n-j+1)\\
 \end{align*}
Substituting the values of $i$ and $j$ obtained above, we have that
\begin{equation*}
\begin{split}
\frac{|\partial L_{n^\ast}|}{|L_{n^\ast}|}  &= 2(n-1)-\frac{(n-j+1)^{2}+(q-1)^{2}(j-i-1)+q^{2}(n-j+1)}{|L_{n^\ast}|}\\
&= \dfrac{-q^4+2q^3+q^2-4{n^{\ast}}^{2}-2q+4n^{\ast}q^2-12n^{\ast}q+8n^{\ast}n-4n^{\ast}}{4n^{\ast}}\\
\end{split}
\end{equation*}

Let \[ g_{1}(n) := (2-\sqrt{2})n,\    g_{2}:=4n^{\ast}\cdot g_{1}(n) \textrm{ and }  h(n):=n-(p+1)/2. \]
We have the following limits.
\[\lim_{n \to \infty} \frac{2(n-1)}{g_{1}(n)}=2+\sqrt{2}; \lim_{n \to \infty} \frac{p(n-p)}{g_{1}(n)h(n)} = \sqrt{2}- 1;  \]
\[ \lim_{n \to \infty}  \frac{(n-1)^2}{g_{1}(n)h(n)}=2; \lim_{n \to \infty} \frac{4n^{\ast}}{g_{2}(n)}=0,
\lim_{n \to \infty} \frac{q^2-2q}{g_{2}(n)}=0; \] 
\[\lim_{n \to \infty}\frac{-q^4-4(n^{\ast})^2+4n^{\ast}q^2}{g_{2}(n)}=0; 
\lim_{n \to \infty}\frac{8n^{\ast}n}{g_{2}(n)}=2+\sqrt{2};\]
\[\lim_{n \to \infty}\frac{-12n^{\ast}q}{g_{2}(n)}=\frac{-3(1+\sqrt{2})}{2}; \textrm{ and }
\lim_{n \to \infty} \frac{2q^3}{g_{2}(n)}=\frac{1}{2\sqrt{2}-2}.\]
It follows that \begin{equation*}
\lim_{n \to \infty}\frac{|\partial F'|}{\left (2-\sqrt{2}\right )n|F'|}=\lim_{n \to \infty}\frac{|\partial L_{n^{\ast}}|}{\left (2-\sqrt{2} \right )n|L_{n^{\ast}}|}=1.
\end{equation*} The theorem follows from Lemma \ref{lower_bound} and the fact that 
$  \operatorname{iso}(J(n,2)) \le |\partial L_{n^{\ast}}|/|L_{n^{\ast}}|.$ 
\end{proof}

To finish the paper we make the following remark.
We have verified by computer that  
\[\left | \frac{|\partial L_{n^\ast}|}{|L_{n^\ast}|}-\frac{|\partial F'|}{|F'|} \right | \le \frac{3}{2},\] 
for all $ n\leq 10^{9} $. We conjecture that this inequality is valid for all natural numbers $ n\geq 3 $.


\bibliographystyle{abbrv}
\bibliography{iso}

\begin{thebibliography}{10}

\bibitem{ahlswede1978graphs}
R.~Ahlswede and G.~Katona.
\newblock Graphs with maximal number of adjacent pairs of edges.
\newblock {\em Acta Mathematica Academiae Scientiarum Hungaricae},
  32(1-2):97--120, 1978.

\bibitem{bey2006remarks}
C.~Bey.
\newblock Remarks on an edge isoperimetric problem.
\newblock In {\em General Theory of Information Transfer and Combinatorics},
  pages 971--978. Springer, 2006.

\bibitem{buser1978cubic}
P.~Buser.
\newblock Cubic graphs and the first eigenvalue of a {Riemann} surface.
\newblock {\em Mathematische Zeitschrift}, 162:87--99, 1978.

\bibitem{buser1984bipartition}
P.~Buser.
\newblock On the bipartition of graphs.
\newblock {\em Discrete applied mathematics}, 9(1):105--109, 1984.

\bibitem{etzion1996chromatic}
T.~Etzion and S.~Bitan.
\newblock On the chromatic number, colorings, and codes of the {Johnson} graph.
\newblock {\em Discrete applied mathematics}, 70(2):163--175, 1996.

\bibitem{goldberg1984minimal}
M.~K. Goldberg and R.~Gardner.
\newblock On the minimal cut problem.
\newblock {\em Progress in Graph Theory}, pages 295--305, 1984.

\bibitem{graham1980lower}
R.~Graham and N.~Sloane.
\newblock Lower bounds for constant weight codes.
\newblock {\em IEEE Transactions on Information Theory}, 26(1):37--43, 1980.

\bibitem{harper1977stabilization}
L.~Harper.
\newblock Stabilization and the edgesum problem.
\newblock {\em Ars Combinatoria}, 4(1):225--270, 1977.

\bibitem{harper1991problem}
L.~Harper.
\newblock On a problem of {Kleitman} and {West}.
\newblock {\em Discrete mathematics}, 93(2-3):169--182, 1991.

\bibitem{johnson1962new}
S.~Johnson.
\newblock A new upper bound for error-correcting codes.
\newblock {\em IRE Transactions on Information Theory}, 8(3):203--207, 1962.

\bibitem{liebler2014neighbour}
R.~A. Liebler and C.~E. Praeger.
\newblock Neighbour-transitive codes in {Johnson} graphs.
\newblock {\em Designs, codes and cryptography}, 73(1):1--25, 2014.

\bibitem{macwilliams1977theory}
F.~J. MacWilliams and N.~J.~A. Sloane.
\newblock {\em The theory of error-correcting codes}, volume~16.
\newblock Elsevier, 1977.

\bibitem{neunhoffer2014sporadic}
M.~Neunh{\"o}ffer and C.~E. Praeger.
\newblock Sporadic neighbour-transitive codes in {Johnson} graphs.
\newblock {\em Designs, codes and cryptography}, 72(1):141--152, 2014.

\end{thebibliography}




\end{document}